\documentclass[3p,11pt,a4paper,oneside,superscriptaddress,noshowpacs,centertags]{article}
 \usepackage{afterpage}
 \usepackage{fancyhdr}
  \usepackage{euscript}
 \usepackage[colorlinks]{hyperref}
 \usepackage{lipsum}% just to generate text for the example
 \newcommand\shorttitle{Clairaut Anti-invariant  Riemannian Map}
 \newcommand\authors{A. Zaidi and G. Shanker}
 
 \usepackage{supertabular, setspace,hyperref}
 \usepackage[a4paper, total={6.5in, 9.5in}]{geometry}
 \usepackage{amsmath}
 \usepackage{mathtools}
 \usepackage{amsfonts}
 \usepackage{amssymb}
 \usepackage{amsthm}
 \usepackage{caption}
 \usepackage{subcaption}
 \usepackage{authblk}
 \usepackage[sort]{cite}
\newtheorem{theorem}{Theorem}[section]

\newtheorem{corollary}[theorem]{Corollary}

\newtheorem{definition}[theorem]{Definition}
\newtheorem{example}[theorem]{Example}

\newtheorem{lemma}[theorem]{Lemma}

 \usepackage{chngcntr}
 
 \counterwithin*{equation}{section}
 \counterwithin*{equation}{section}
 
\begin{document}
 	\afterpage{\cfoot{\thepage}}
 	\clearpage
 	\date{}

 	\title{\textbf {Clairaut anti-invariant Riemannian maps to trans-Sasakian manifolds}}

 	\author{Adeeba Zaidi}
 		\author{Gauree Shanker\footnote{corresponding author, Email: gauree.shanker@cup.edu.in}}
 	 
 	\affil{\footnotesize Department of Mathematics and Statistics,
 		Central University of Punjab, Bathinda, Punjab-151 401, India.\\
 		Email:adeebazaidi.az25@gmail.com, gauree.shanker@cup.edu.in*}
 	\maketitle
 	%	\maketitle
 	%	\begin{center}
 	%		\author{\textbf{Gauree Shanker, Sarita Rani\footnote{Corresponding author}}}
 	%	\end{center}
 	%	\begin{center}
 	%		Department of Mathematics and Statistics\\
 	%%		School of Basic and Applied Sciences\\
 	%		Central University of Punjab, Bathinda, Punjab-151 001, India\\
 	%		Email:  grshnkr2007@gmail.com, saritas.ss92@gmail.com
 	%	\end{center}
	\begin{abstract}
  		In this article, we introduce Clairaut anti-invariant Riemannian maps from Riemannian manifolds to trans-Sasakian manifolds. We derive necessary and sufficient condition for an anti-invariant map to be Clairaut when base manifold is trans-Sasakian manifold. We discuss the integrability of $range\pi_{*}$ and $(range\pi_{*})^\perp$. Further, we establish harmonicity of these map. Finally, we construct a nontrivial example of such maps for justification. 
  	\end{abstract}
  	\vspace{0.2 cm}
  	
  	 	\begin{small}
  	 	\end{small}
  	 
  	 	
  	\maketitle

  	\section{Introduction}
  	The concept of Riemannian maps between Riemannian manifolds was firstly introduced by Fischer in 1992 (\cite{F}). He described Riemannian maps as the generalization of isometric immersions, Riemannian submersions and isometries. The interesting part of Riemannian maps is that they satisfy general eikonal equation which is a bridge between geometrical and physical aspects of optics. In \cite{F}, Fischer described: let	$\pi: (M,g_1)\rightarrow (B,g_2)$	be a smooth map between smooth finite dimensional Riemannian manifolds $(M,g_1)$ and $(B,g_2)$ such that $0<rank\pi<min \{dimM,dimB\}$. Let $\pi_{*p}: T_pM\rightarrow T_{\pi(p)}B$ denotes the differential map at $p \in  M$, and $ \pi(p)  \in  B$. Then $T_pM$ and $T_{\pi(p)}B$ split
  	  	 		 	orthogonally with respect to $ g_1(p)$ and $g_2(\pi(p))$, respectively, as (\cite{F})
  	  	 		 	 	\begin{equation*} 
  	  	 		 	 		\begin{split}
  	  	 		 	 			T_pM &= ker\pi_{*p} \oplus (ker \pi_{*p})^\perp,\\
  	  	 		 	 			&=\mathcal{V}_p\oplus \mathcal{H}_p,
  	  	 		 	 		\end{split}	
  	  	 		 	 	\end{equation*}
  	   		 	 		\begin{equation*}
  	   		 	 		T_{\pi(p)}B = range\pi_{*p}  \oplus  (range \pi_{*p})^\perp,
  	   		 	 	\end{equation*}
  	  	 		 	 where $\mathcal{V}_p= ker\pi_{*p}$ and $\mathcal{H}_p=(ker \pi_{*p})^\perp$ are vertical and horizontal parts of $T_{p}M$ respectively. The map $\pi$ is called Riemannian map at
  	  	 		 	 	$p  \in  M,$ if the horizontal restriction
  	  	 		 	 	\begin{equation*}
  	  	 		 	 		(\pi_{*p})^h = \pi_{*p}\ |\ _{\mathcal{H}_p} :\mathcal{H}_p\rightarrow range\pi_{*p}
  	  	 		 	 	\end{equation*}
  	  	 		 	  is a linear isometry between the spaces $(ker\pi_{*p},g_1\ |_ {ker\pi_{*p}})$ and $(range\pi_{*p},g_2|_{(range\pi_{*p})}).$ In other words, $(\pi_{*p})^h$ satisfies the equation   
  		\begin{equation}\label{19}
  	  	 		 	 		g_2(\pi_{*} W,\pi_{*} Z) \ =\ g_1(W,Z),
  	  	 		 	 	\end{equation}
  	  	 		 	 	for all vector fields $W, Z$ tangent to $\Gamma (ker\pi_{*p})^\perp $.
  	  	 		 	 	During the last three decades, many authors have studied Riemannian maps (\cite{S3, S4, S1, PS ,SB ,AZ}) and the investigation is still going on.\\
  	  	 		 	  Clairaut's theorem plays an important role in differential geometry, which states that for any geodesic $\gamma$ on a surface of revolution, the function $ rsin\theta$ is constant, where $r$ is the distance between a point on the surface and rotation axis, whereas $\theta$ is an angle between $\gamma$ and the meridian curve through $\gamma$. Inspired by this theorem, Bishop (\cite{Bi}) introduced Clairaut Riemannian submersion and derived the necessary and sufficient conditions for a submersion to be Clairaut Riemannian submersion. Hereafter, Riemannian submersion is investigated broadly in both hermitian geometry as well as contact geometry \cite{P1,P2,H2,GY}. \c{S}ahin \cite{S7,S8} investigated Clairaut conditions on Riemannian map. Later, various types of Clairaut Riemannan maps  such as invariant, anti-invariant, semi-invaiant are studied with Kahler structure and cosymplectic structure  \cite{S8,AK, AK1,YR,PK}.\\
The notion of a trans-Sasakian structure $(\psi,\xi,\eta,g,\alpha,\beta)$ on $(2n+1)$-dimension manifold $B$, was introduced by Oubina (\cite{Ou}) which can be seen as a generalization of Sasakian, Kenmotsu and cosymplectic structure on a contact metric manifold, where $\alpha,\beta$ are smooth functions on $B$ . Generally, a trans-Sasakian manifold $(B,\psi,\xi,\eta,g,\alpha,\beta)$ is called a trans-Sasakian manifold of type $(\alpha,\beta)$ and manifolds of type $(\alpha, 0),(0, \beta)$ and $(0, 0)$ are called, $\alpha-$Sasakian, $\beta-$Kenmotsu and cosymplectic manifolds respectively. Since the geometry of Sasakian manifolds is very rich, it would be interesting to study different type of Riemannian maps on this structure. In this paper, we study Clairaut anti-invariant Riemannian maps from Riemannian manifolds to trans-Sasakian manifolds. The paper is organized as follows: In section $2$, we give all the basic definitions and terminologies, needed throughout the paper. In section $3$, we introduced Clairaut anti-invariant Riemannian maps from Riemannian manifolds to trans-Sasakian manifolds admitting horizontal Reeb vector field. Further, we study necessary and  sufficient condition for a curve on base manifold to be geodesic and obtain necessary and sufficient condition for an anti-invariant Riemannian map to be Clairaut when base manifold is trans-Sasakian with horizontal Reeb vector field. We find the integrability condition for distributions of tangent bundle on base manifold. Later, we check the harmonicity of these maps. We also construct some nontrivial examples for such maps.
     
 	\section{Preliminaries}
 	In this section, we recall the definitions of contact manifolds, trans-Sasakian manifolds and some important properties related to Riemannian maps.\\
 		Let $B$ be a $(2n+1)-$dimensional differentiable manifold, then $B$ is said to have an almost-contact structure $(\psi,\xi,\eta)$, if it admits a $(1,1)$ tensor field $\psi$, a vector field called characteristic vector field or Reeb vector field $\xi$, and a  $1-$form $\eta$, satisfying (\cite{B})
 		\begin{equation} \label{a}
 			\psi^2=-I + \eta \otimes \xi,~~~ \psi \xi=0,~~\eta\circ\psi=0,~~ \eta(\xi)=1,  
 		\end{equation}
 		where $I$ is the identity mapping. A Riemannian metric $g$ on an almost-contact manifold $B$ is said  to be compatible with the almost-contact structure $(\psi, \xi, \eta)$, if for any vector fields $W,Z \in \varGamma(TB)$, $g$ satisfies (\cite{B})
 		\begin{align}
 			g(\psi W,\psi Z) \ &=\  g(W,Z)-\eta (W)\eta (Z),  \label{b} \\
 				g(\psi W,Z) &= g(W,\psi Z), \ \eta(W) \ =\ g(W,\xi), \label{c}
 		\end{align}
 		 the structure $(\psi, \xi, \eta,g)$ is called an almost contact metric structure.\\
 		 The almost contact structure $(\psi, \xi, \eta)$ is said to be normal if $ N + d\eta \otimes \xi=0,$ where $N$ is the
 		Nijenhuis tensor of $\psi$. If $d\eta=\Phi$, where $\Phi(W,Z)=g(\psi W,Z)$ is a tensor field of type $(0,2)$, then an almost contact metric
 		structure is said to be normal contact metric structure. An almost contact metric manifold $B$
 		is called a trans-Sasakian manifold of type $(\alpha,\beta)$  (\cite{B}), if it satisfies
 		\begin{align}
 					(\nabla_W\psi)Z \ &= \ \alpha(g(W,Z)\xi-\eta(Z)W)+\beta[g(\psi W,Z)\xi-\eta(Z)\psi W],\label{4} \\  
 				(\nabla_W\eta)Z \ &= \ -\alpha g(\psi W,Z)\xi+\beta g(\psi W,\psi Z),\\
 				\nabla_{W}\xi&=-\alpha \psi W +\beta (W-\eta(W)\xi), \label{e}	
 		\end{align}
 		 where $\alpha, \ \beta$ are smooth functions and $\nabla$ is Levi-Civita connection of $g$ on $B$. Further, it can be seen that a trans-Sasakian manifold of type $(\alpha, 0)$ is a $\alpha-$Sasakian manifold and a trans-Sasakian manifold of type $(0, \beta)$ is a $\beta-$Kenmotsu manifold. A trans-Sasakian manifold of type (0, 0) is called a cosymplectic manifold. In particular, for $\alpha=1, \beta=0$; and $\alpha=0, \beta=1$, a trans-Sasakian manifold will be Sasakian and Kenmotsu manifold respectively. \begin{example}\cite{B}
 	Let	$B=\{(u,v,w)\in\mathbb{R}^3,w\neq0\}$ be a $3-$dimensional Riemannian manifold associated with Riemannian metric $g_2$ given by
 	$$  g_2=\frac{1}{4}\begin{pmatrix}
 	 		 		 1+v^2&0&-v\\
 	 		 		 0&1&0\\
 	 		 		 -v&0&1
 	 		 		 \end{pmatrix},$$
 	$1-$form $\eta=\frac{1}{2}(dw-vdu)$  and linearly independent global frame $\{E_1,E_2,E_3\}$ be defined as $E_1=2\frac{\partial}{\partial v}, E_2=\psi E_1=2(\frac{\partial}{\partial u}+v\frac{\partial}{\partial w}), E_3=2\frac{\partial}{\partial w}=\xi,$ where $\xi$ is the characteristic vector field (Reeb vector field) and the $(1,1)-$ tensor field $\psi$ is given by the matrix 
 		 		  $$\psi=\begin{pmatrix}
 		 		0&1&0\\
 		 		-1&0&0\\
 		 		0&v&0
 		 		 \end{pmatrix}, 
 		 		$$ then $B$ is a trans-Sasakian manifold of type $(1,0)$.
 		 \end{example} 
 		 
 Further, let $\pi: (M^m,g_1)\rightarrow (B^b,g_2)$ be a smooth map between smooth finite dimensional Riemannian manifolds, then the differential map $\pi_{*}$ of $\pi$ can be viewed as a section
 		 of bundle $Hom(TM, \pi^{-1}TB)\rightarrow M$, where $\pi^{-1}TB$ is the pullback bundle whose fibers at $p \in M$ is $(\pi^{-1}TB)_{p} = T_{\pi(p)}B$. The bundle $Hom(TM, \pi^{-1}TB)$
 		 has a connection $\nabla$ induced from the Levi-Civita connection $\nabla^{M}$ and the pullback connection $\overset{B}{\nabla^{\pi}}$, then the second fundamental form of $\pi$ is given by (\cite{S7} )
 		 \begin{equation}  \label{5}
 		 			(\nabla \pi_{*})(W,Z) = \overset{B}{\nabla^{\pi}_{W}}\pi_{*}Z-\pi_{*}(\nabla_W^MZ)
 		\end{equation}
 		 	for all $W,Z ~\in \Gamma (TM)$ and $\overset{B}{\nabla^{\pi}_{W}}\pi_{*}Z \circ \pi = \nabla ^{B}_{\pi_{*}W}\pi_{*}Z$.\\ \\ 
 Also, for any vector field $W$ on $M$ and any section $V$ of $(range\pi{*})^\perp$, we have $\nabla ^{\pi\perp}_{W} V$, the
 		 	 	 orthogonal projection of $\nabla ^{B}_{W}V $ on $(range\pi_{*})^\perp$, where $\nabla \pi_{*}^\perp $ is linear
 		 	 	connection on $(range\pi_{*})^\perp$ such that $\nabla ^{\pi\perp} g_{2} = 0.$\\
 		 	 	 Now, for a Riemannian map, we have (\cite{S7})
 		 	 	\begin{equation}\label{6}
 		 	 	 \nabla ^{B}_{\pi_{*}W}V=-\mathcal{A}_{V}\pi_{*}W+\nabla^{\pi\perp}_WV,
 		 	 	\end{equation}
 		 	 	where $\mathcal{A}_V\pi_{*}W$ is the tangential component of  $\nabla ^{B}_{\pi_{*}W}V$. At $p\in M$, we have $\nabla ^{B}_{\pi_{*}W}V(p)\in T_{\pi(p)}B,~ ~\mathcal{A}_{V}\pi_{*}W(p) \in \pi_{*p}(T_pM)$ and $\nabla^{\pi\perp}_WV(p)\in (\pi_{*p}(T_pM))^\perp.$ It is easy to see that $\mathcal{A}_{V}\pi_{*}W$ is bilinear in $V$ and $\pi_{*}W$, and $\mathcal{A}_{V}\pi_{*}W$ at $p$ depends only on $V_{p}$ and $\pi_{*p}W_{p}.$ By direct computations, we obtain \cite{S7}
 		 	 \begin{equation}\label{10}
 		 	 	g_2 (\mathcal{A}_V \pi_{*}W,\pi_{*}Z) = g_2 (V,(\nabla\pi_{*})(W, Z))
 		 	 \end{equation} 
 		 	 	for $W, Z\in \Gamma ((ker\pi_{*})^\perp)$ and $V\in \Gamma((range\pi_{*})^\perp)$. Since $(\nabla\pi_{*})$ is symmetric, it follows
 		 	 	that $\mathcal{A}_V$ is a symmetric linear transformation of $range\pi_{*}$.\\
 		 	 Moreover, Let $\pi: (M, g_1)\rightarrow (B, g_2)$ be a Riemannian map between
 		 	 Riemannian manifolds, then $\pi$ is umbilical Riemannian map if and only if \cite{S7}
 		 	 \begin{equation}\label{7}
 		 	(\nabla \pi_*)(W, Z ) = g_1(W,Z)H_2
 		 	 \end{equation}
 		 	for $ W, Z \in \Gamma((ker\pi_{*})^\perp$ and $H_2$ is nowhere zero vector field on $(range\pi_*)^\perp$. 
 		 	Also, the mean curvatures of vertical distribution $\mathcal{V}$ and horizontal distribution $\mathcal{H}$ are defined as (\cite{S7}) 
 		 	\begin{equation}\label{22}
 		 	\varrho^{\mathcal{V}}=\dfrac{1}{q}\sum_{i=1}^{q}\mathcal{H}(\nabla_{e_i}e_i),~~~~	\varrho^{\mathcal{H}}=\dfrac{1}{m-q}\sum_{j=1}^{m-q}\mathcal{V}(\nabla_{E_j}E_j),
 		 	\end{equation}
 		 	where $\{e_i\}_{i=1}^q$ and $\{E_j\}_{j=1}^{m-q}$ are local frames of $\mathcal{V}$ and $\mathcal{H}$ respectively. A distribution on $M$ is said to be minimal, if for each point in $M$, the mean curvature vanishes. 
 		 	\begin{lemma}\cite{S7}
 		 	Let $\pi$ be a Riemannian map from a Riemannian manifold
 		 	$(M, g_1 )$ to a Riemannian manifold $(B, g_2 )$. Then, $\forall~ W, Y, Z \in\Gamma((ker\pi_{*})^\perp)$, we
 		 	have
 		 \begin{equation*}
 		 	g_2 ((\nabla \pi_{*})(W, Y), \pi_{*}Z) = 0.
 		 \end{equation*}
 		 	\end{lemma}
 		 	
 		\begin{definition}\cite{S7}
 		A Riemannian map $\pi: (M, g_1)\rightarrow (B, g_2)$ between Riemannian manifolds
 	is called Clairaut Riemannian map if there is a function $r : B \rightarrow R^+$ such that for every geodesic $\Omega$ on $B$, the function $(r\circ\Omega) sin \theta(s)$ is
 		constant, where, $\pi_{*}Z \in \Gamma(range\pi_{*})$ and $V\in \Gamma(range\pi_{*})^\perp$are components of $\dot{\Omega}(s)$, and $\theta(s)$ is the angle between $\dot{\Omega}(s)$ and $V$.
 		\end{definition} 
 		
 		\begin{definition}\cite{AK}
Let $\pi : (M, g_1) \rightarrow (B, g_2)$ be a Riemannian map between Riemannian manifolds such that $range\pi_{*}$ is connected and $(range\pi_*)^\perp$ is totally geodesic, and $\gamma,\Omega = \pi\circ\gamma$ be geodesics on $M$ and $B$ respectively.
Then, $\pi$ is a Clairaut Riemannian map with $r = e^h$ if and only if $\pi$ is umbilical map, and has $H = -\nabla^Bh$, where $h$ is a smooth function on $B$ and $ H$ is the mean curvature vector field of $range\pi_{*}$.

 		\end{definition}				 	 	
 		 
 	\section{
 	Clairaut anti-invariant Riemannian maps to trans-Sasakian manifolds}
 In this section, we define Clairaut anti-invariant Riemannian map from a Riemannian manifold to a trans-Sasakian manifold and discuss the geometry of such maps. Throughout this section, we are considering Reeb vector field in horizontal space $\mathcal{H}$ of $TM$ and  $(range\pi_{*x})^\perp$ as totally geodesic distribution. 	
 \begin{definition} 
 Let $\pi$ be a Riemannian map from a Riemannian manifold
 $(M, g_1)$ to an almost contact metric manifold $(B,g_2,\psi,\eta,\xi)$. Then we say that $\pi$ is
 an anti-invariant Riemannian map at $p \in M$ if $\psi (range\pi_{*p})\subset (range\pi_{*p})^\perp$. If
 $\pi$ is an anti-invariant Riemannian map for all $p \in M$, then $\pi$ is called an
 anti-invariant Riemannian map.
 \end{definition}
 In this case, the horizontal distribution $(range\pi_{*})^\perp$ can be  decomposed as
 \begin{equation}
(range\pi_{*})^\perp
 = \psi range\pi_{*} \oplus \mu,
 \end{equation}
 where $\mu$ is the orthogonal complementary distribution of $ \psi range\pi_{*}$ in $(range\pi_{*})^\perp$
 and also invariant with respect to $\psi$. $pi$ admits vertical Reeb vector field if $\xi\in range\pi_{*}$ whereas if it admits horizontel Reeb vector field $\xi$, then $\xi\in (range\pi_{*})^\perp$. It is clear that in case of horizontal Reeb vector fields, $\xi\in \mu$. For any $V\in (range\pi_{*})^\perp$, we can have
 \begin{equation}\label{8}
 \psi V=\mathcal{B}V+\mathcal{C}V,
 \end{equation}
  	where $\mathcal{B}V\in\Gamma(range\pi_{*})$ and $\mathcal{C}V\in \Gamma((range\pi_{*})^\perp)$.  
  	
  	\begin{definition}
  	An anti-invariant Riemannian map from a Riemannian manifold to a contact manifold is said to be Clairaut if it satisfies Definition 2.1.
  	\end{definition} 		

 \begin{theorem} %\label{1}{\cite{R2}}
 Let $\pi: M^m\rightarrow B^b$ be an anti-invariant Riemannian map from a Riemannian manifold $(M^m,g_1)$ to a trans-Sasakian manifold $(B^b,g_2,\psi,\eta,\xi)$ of type $(\alpha,\beta)$ with horizontal Reeb vector field $\xi$ and $\gamma:J\subset\mathbb{R}\rightarrow M$ be a geodesic curve on $M$, then the curve $\Omega =\pi\circ\gamma$ is a geodesic on $B$ if and only if
 	\begin{equation} \label{9}
 	-\mathcal{A}_{\psi \pi_{*}W}\pi_{*}W	-\mathcal{A}_{\mathcal{C}U}\pi_{*} W+\pi_{*}(\mathcal{H}\nabla_{W}^MZ)+\nabla_{U}^B\pi_{*}Z
 	+\eta(U)[\alpha \pi_{*}W+\beta\pi_{*}Z]=0,
 	\end{equation}
 	\begin{multline}\label{10}
 	\nabla_W^{\pi\perp}\psi(\pi_{*}W)+\nabla_W^{\pi \perp}\mathcal{C}Z+\nabla_U^{\pi\perp}\psi(\pi_{*}W)+\nabla_U^{\pi\perp}\mathcal{C}U+(\nabla\pi_{*})(W,Z)-\alpha||\dot{\Omega}||^2\xi+\eta(U)[\alpha U\\ +\beta(\psi(\pi_{*}W)+\mathcal{C}U)]=0
 	\end{multline} 
 	for any  $U\in \Gamma((rangeF_{*})^{\perp})$ and $W,Z~ \in \Gamma ((kerF_{*})^\perp)$ such that $\pi_{*}Z=\mathcal{B}U$ with $\pi_{*}W$ and $U$ as vertical and horizontal components of $\dot{\Omega}(s).$
 \end{theorem}
 
 \begin{proof}
 
 Let  $U\in \Gamma((rangeF_{*})^{\perp})$ and $W~ \in \Gamma ((kerF_{*})^\perp)$ such that $\dot{\Omega}(s)=\pi_{*}W(s)+U(s)$. Since $B$ is a trans-Sasakian manifold, using
\eqref{4}, we get
 \begin{align*}
\psi\nabla^B_{\dot{\Omega}}\dot{\Omega}=\nabla^B_{\dot{\Omega}}\psi\dot{\Omega}-\alpha[g_2(\dot{\Omega},\dot{\Omega})\xi-\eta(\dot{\Omega})\dot{\Omega}]-\beta[g_2(\psi\dot{\Omega},\dot{\Omega})\xi-\eta(\dot{\Omega})\psi\dot{\Omega}].
 \end{align*}
 Since $\dot{\Omega}=\pi_{*}W+U$, the above equation can be rewritten as
 
 \begin{multline}\label{a}
\psi\nabla^B_{\dot{\Omega}}\dot{\Omega}=\nabla_{\pi_{*}W}^B\psi\pi_{*}W+\nabla_{\pi_{*}W}^B\psi U+\nabla_{U}^B\psi\pi_{*}W+\nabla_{U}^B\psi U-\alpha[||\dot{\Omega}||^2\xi-\eta(U)\pi_{*}W-\eta(U)U]\\+\beta[\eta(U)\psi(\pi_{*}W)+\eta(U)\psi U].
 \end{multline}
Using \eqref{6} and \eqref{8} in \eqref{a}, we get
\begin{equation}\label{b}
\begin{split}
\psi\nabla^B_{\dot{\Omega}}\dot{\Omega}&=-\mathcal{A}_{\psi\pi_{*}W}\pi_{*}W+\nabla^{\pi\perp}_W\psi(\pi_{*}W)-\mathcal{A}_{\mathcal{C}U}\pi_{*}W+\nabla_W^{\pi\perp}\mathcal{C}U+\nabla_{\pi_{*}W}^B\mathcal{B}U+\nabla_U^B\psi(\pi_{*}W)\\&+\nabla_U^B\mathcal{B}U+\nabla_U^B\mathcal{C}U-\alpha[||\dot{\Omega}||^2\xi-\eta(U)\pi_{*}W-\eta(U)U]+\beta[\eta(U)\psi(\pi_{*}W)+\eta(U)\psi U],\\
&=-\mathcal{A}_{\psi\pi_{*}W}\pi_{*}W-\mathcal{A}_{\mathcal{C}U}\pi_{*}W+\nabla^{\pi\perp}_W\psi(\pi_{*}W)+\nabla_W^{\pi\perp}\mathcal{C}U+\nabla_U^{\pi\perp}\psi(\pi_{*}W)+\nabla_U^{\pi\perp}\mathcal{C}U\\&+\nabla_{\pi_{*}W}^B\mathcal{B}U+\nabla_U^B\mathcal{B}U-\alpha[||\dot{\Omega}||^2\xi-\eta(U)\pi_{*}W-\eta(U)U]+\beta[\eta(U)\psi(\pi_{*}W)+\eta(U)\psi U].
\end{split}
\end{equation}
 Since $g(\mathcal{B}U,V)=0$ for any $V\in\Gamma((range\pi_{*})^\perp)$, therefore $\nabla_U^B\mathcal{B}U\in \Gamma(range\pi_{*})$. Let $Z\in\Gamma((ker\pi)^\perp)$ such that $\ \pi_{*}Z=\mathcal{B}U$, then using \eqref{5} in \eqref{b}, we have
 \begin{equation}
\begin{split}
\psi\nabla^B_{\dot{\Omega}}\dot{\Omega}&=-\mathcal{A}_{\psi\pi_{*}W}\pi_{*}W-\mathcal{A}_{\mathcal{C}U}\pi_{*}W+\nabla^{\pi\perp}_W\psi(\pi_{*}W)+\nabla_W^{\pi\perp}\mathcal{C}U+\nabla_U^{\pi\perp}\psi(\pi_{*}W)+\nabla_U^{\pi\perp}\mathcal{C}U\\&+(\nabla\pi_{*})(W,Z)+\pi_{*}(\mathcal{H}\nabla^M_WZ)+\nabla^B_U\pi_{*}Z-\alpha[||\dot{\Omega}||^2\xi-\eta(U)\pi_{*}W-\eta(U)U]\\&+\beta\eta(U)\psi(\pi_{*}W)+\beta\eta(U)\psi U.
\end{split}
 \end{equation} 
Since, $\Omega$ is geodesic, $\nabla_{\dot{\Omega}}^B\dot{\Omega}=0$. Separating vertical and horizontal parts of the above equation we get \eqref{9} and \eqref{10}.
%\begin{align}
%\begin{split}
%-\mathcal{A}_{\psi\pi_{*}W}\pi_{*}W-\mathcal{A}_{\mathcal{C}U}\pi_{*}W+\pi_{*}(\mathcal{H}\nabla^M_WZ)+\nabla^B_U\pi_{*}Z+\eta(U)[\alpha\pi_{*}W+\beta\pi_{*}Z]=0\\
%\nabla^{\pi\perp}_W\psi(\pi_{*}W)+\nabla_W^{\pi\perp}\mathcal{C}U+\nabla_U^{\pi\perp}\psi(\pi_{*}W)+\nabla_U^{\pi\perp}\mathcal{C}U+(\nabla\pi_{*})(W,Z)-\alpha||\dot{\Omega}||^2\xi+\eta(U)[\alpha U\\+\beta(\psi(\pi_{*}W)+\mathcal{C}U)]=0
%\end{split}
%\end{align} 
 \end{proof}
 
 \begin{corollary}
Let $\pi: M^m\rightarrow B^b$ be an anti-invariant Riemannian map from a Riemannian manifold $(M^m,g_1)$ to a trans-Sasakian manifold $(B,g_2,\psi,\eta,\xi)$ of type $(\alpha,0)$ with horizontal Reeb vector field $\xi$ and $\gamma:J\subset\mathbb{R}\rightarrow M$ be a geodesic on $M$, then the curve $\Omega =\pi\circ\gamma$ is a geodesic on $B$ if and only if
	\begin{equation*}
 	-\mathcal{A}_{\psi \pi_{*}W}\pi_{*}W	-\mathcal{A}_{\mathcal{C}U}\pi_{*} W+\pi_{*}(\mathcal{H}\nabla_{W}^MZ)+\nabla_{U}^B\pi_{*}Z
 	+\eta(U)\alpha \pi_{*}W=0,
 	\end{equation*}
 	\begin{multline*}
 	\nabla_W^{\pi\perp}\psi(\pi_{*}W)+\nabla_W^{\pi \perp}\mathcal{C}Z+\nabla_U^{\pi\perp}\psi(\pi_{*}W)+\nabla_U^{\pi\perp}\mathcal{C}U+(\nabla\pi_{*})(W,Z)+\alpha[\eta(U)U-||\dot{\Omega}||^2\xi]=0
 	\end{multline*} 
 	for any  $U\in \Gamma((rangeF_{*})^{\perp})$ and $W,Z~ \in \Gamma ((kerF_{*})^\perp)$ such that $\pi_{*}Z=\mathcal{B}U$ with $\pi_{*}W$ and $U$ as vertical and horizontal components of $\dot{\Omega}(s).$
 \end{corollary}

\begin{corollary}
Let $\pi: M^m\rightarrow B^b$ be an anti-invariant Riemannian map from a Riemannian manifold $(M^m,g_1)$ to a trans-Sasakian manifold $(B,g_2,\psi,\eta,\xi)$ of type $(0,\beta)$ with horizontal Reeb vector field $\xi$ and $\gamma:J\subset\mathbb{R}\rightarrow M$ be a geodesic on $M$, then the curve $\Omega =\pi\circ\gamma$ is a geodesic on $B$ if and only if
	\begin{equation*}
 	-\mathcal{A}_{\psi \pi_{*}W}\pi_{*}W	-\mathcal{A}_{\mathcal{C}U}\pi_{*} W+\pi_{*}(\mathcal{H}\nabla_{W}^MZ)+\nabla_{U}^B\pi_{*}Z
 	+\beta\eta(U)\pi_{*}Z]=0,
 	\end{equation*}
 	\begin{multline*}
 	\nabla_W^{\pi\perp}\psi(\pi_{*}W)+\nabla_W^{\pi \perp}\mathcal{C}Z+\nabla_U^{\pi\perp}\psi(\pi_{*}W)+\nabla_U^{\pi\perp}\mathcal{C}U+(\nabla\pi_{*})(W,Z)+\beta\eta(U)[\psi(\pi_{*}W)+\mathcal{C}U]=0
 	\end{multline*} 
 	for any  $U\in \Gamma((rangeF_{*})^{\perp})$ and $W,Z~ \in \Gamma ((kerF_{*})^\perp)$ such that $\pi_{*}Z=\mathcal{B}U$ with $\pi_{*}W$ and $U$ as vertical and horizontal components of $\dot{\Omega}(s).$
\end{corollary}
 
 \begin{corollary}
 Let $\pi: M^m\rightarrow B^b$ be an anti-invariant Riemannian map from a Riemannian manifold $(M^m,g_1)$ to a trans-Sasakian manifold $(B,g_2,\psi,\eta,\xi)$ of type $(0,0)$ with horizontal Reeb vector field $\xi$ and $\gamma:J\subset\mathbb{R}\rightarrow M$ be a geodesic on $M$, then the curve $\Omega =\pi\circ\gamma$ is a geodesic on $B$ if and only if
 	\begin{equation*}
  	-\mathcal{A}_{\psi \pi_{*}W}\pi_{*}W	-\mathcal{A}_{\mathcal{C}U}\pi_{*} W+\pi_{*}(\mathcal{H}\nabla_{W}^MZ)+\nabla_{U}^B\pi_{*}Z=0,
  	\end{equation*}
  	\begin{equation}
  	\nabla_W^{\pi\perp}\psi(\pi_{*}W)+\nabla_W^{\pi \perp}\mathcal{C}Z+\nabla_U^{\pi\perp}\psi(\pi_{*}W)+\nabla_U^{\pi\perp}\mathcal{C}U+(\nabla\pi_{*})(W,Z)=0
  	\end{equation} 
  	for any  $U\in \Gamma((rangeF_{*})^{\perp})$ and $W,Z~ \in \Gamma ((kerF_{*})^\perp)$ such that $\pi_{*}Z=\mathcal{B}U$ with $\pi_{*}W$ and $U$ as vertical and horizontal components of $\dot{\Omega}(s).$
 \end{corollary}
  
 \begin{theorem}
 Let $\pi: M^m\rightarrow B^b$ be an anti-invariant Riemannian map from a Riemannian manifold $(M^m,g_1)$ to a trans-Sasakian manifold $(B^b,g_2,\psi,\eta,\xi)$ of type $(\alpha,\beta)$ with horizontal Reeb vector field $\xi$. Let $\gamma$ and $\Omega$ be geodesics on $M$ and $B$ respectively. Then $\pi$ is a Clairaut anti-invariant Riemannian map with $r=e^h$ if and only if
\begin{equation}\label{ad3}
\begin{split}
g_2(\pi_{*}W,\pi_{*}W)\frac{d(h\circ \Omega)}{ds}&=g_2(\mathcal{A}_{\psi\pi_{*}W}\pi_{*}W,\pi_{*}Z)-g_2(\nabla_W^{\pi\perp}\psi(\pi_{*}W)+\nabla_U^{\pi\perp}\psi(\pi_{*}W),\mathcal{C}U)\\&-\eta(U)[\alpha g_1(W,Z)+\beta||\psi U||^2],
\end{split}
\end{equation}
where $U\in \Gamma((range\pi_{*})^\perp)$ and $W,Z\in \Gamma((ker\pi_{*})^\perp)$ such that $\pi_{*}Z=\mathcal{B}U$. Also $\pi_{*}W$ and $U$ are vertical and horizontal part of $\dot{\Omega}(s)$ respectively and $h$ is a smooth function on $B$. 
 \end{theorem}
 \begin{proof}
 Let $\gamma:J\subset \mathbb{R}\rightarrow M$ and $\Omega=\pi\circ \gamma$ be geodesics on $M$ and $B$ respectively such that $\dot{\Omega}(s)=\pi_{*}W(s)+U(s)$, where $\pi_{*}W\in \Gamma(range\pi_{*})$ and $U\in\Gamma((range\pi_{*})^\perp)$.\\
 Considering $||\dot{\Omega}(s)||^2=c$, we have
 \begin{eqnarray}
g_{2\dot{\Omega}(s)}(U,U)=c cos^2\theta(s),\label{c}\\
g_{2\dot{\Omega}(s)}(\pi_{*}W,\pi_{*}W)=c sin^2\theta(s), \label{d}
 \end{eqnarray}
where $\theta(s)\in[0,\pi]$ is the angle between $\dot{\Omega}$ and $U$. Differentiating \eqref{c} along $\Omega$, we get
\begin{equation}\label{e}
\frac{d}{ds}g_2(U,U)=2g_2(\nabla_{\dot{\Omega}}^BU,U)=-2ccos\theta sin\theta \frac{d\theta}{ds}.
\end{equation} 
 Since $B$ is a trans-Sasakian manifold, \eqref{e} can be written as
  \begin{equation}\label{f}
  2g_2(\nabla_{\dot{\Omega}}\psi U,\psi U)=2g_2(\nabla_{\pi_{*}W+U}^B\psi U,\psi U)=-2ccos\theta sin\theta \frac{d\theta}{ds}.
   \end{equation}
   Now, using \eqref{8} and \eqref{6} in \eqref{f}, we get
 \begin{equation}\label{g}
g_2(\nabla_{\pi_{*}W}^B\mathcal{B}U-\mathcal{A}_{\mathcal{C}U}\pi_{*}W+\nabla_{W}^{\pi\perp}\mathcal{C}U+\nabla_U^B\mathcal{B}U+\nabla_U^\perp \mathcal{C}U,\psi U)=-ccos\theta sin\theta \frac{d\theta}{ds}.
 \end{equation}
Let $Z\in\Gamma((ker\pi_{*})^\perp)$ such that $\pi_{*}Z=\mathcal{B}U$, using \eqref{5} in \eqref{g} we have
\begin{equation}\label{h}
g_2((\nabla^B\pi_{*})(W,Z)+\pi_{*}(\nabla^M_WZ)-\mathcal{A}_{\mathcal{C}U}\pi_{*}W+\nabla_{W}^{\pi\perp}\mathcal{C}U+\nabla_U^B\pi_{*}Z+\nabla_U^\perp \mathcal{C}U,\psi U)=-ccos\theta sin\theta \frac{d\theta}{ds}.
\end{equation}
Using \eqref{9} and \eqref{10} in \eqref{h} and simplifying, we obtain
\begin{multline}\label{i}
-ccos\theta sin\theta \frac{d\theta}{ds}=g_2(\mathcal{A}_{\psi\pi_{*}W}\pi_{*}W,\pi_{*}Z)-g_2(\nabla_W^{\pi\perp}\psi(\pi_{*}W)+\nabla^{\pi\perp}_U\psi(\pi_{*}W),\mathcal{C}U)\\-\alpha\eta(U)g_2(W,Z)-\beta\eta(U)||\psi U||^2.
\end{multline}
Further, $\pi$ is a Clairaut Riemannian map with $r=e^h$ if and only if $$\frac{d}{ds}(e^{h\circ \Omega}sin\theta)=0,$$ This implies,
$$e^{h\circ\Omega}sin\theta\frac{d(h\circ\Omega)}{ds}+e^{h\circ\Omega}cos\theta\frac{d\theta}{ds}=0,$$
\begin{equation}\label{j}
csin^2\theta\frac{d(h\circ\Omega)}{ds}=-csin\theta cos\theta\frac{d\theta}{ds}.
\end{equation}
From \eqref{d}, \eqref{i} and \eqref{j}, we get \eqref{ad3}. 
 \end{proof}
\begin{corollary}
 Let $\pi: M^m\rightarrow B^b$ be an anti-invariant Riemannian map from a Riemannian manifold $(M^m,g_1)$ to a trans-Sasakian manifold $(B^b,g_2,\psi,\eta,\xi)$ of type $(\alpha,0)$ with horizontal Reeb vector field $\xi$$\xi$. Let $\gamma$ and $\Omega$ be geodesic on $M$ and $B$ respectively. Therefore $\pi$ is a Clairaut anti-invariant Riemannian map with $r=e^h$ if and only if
 \begin{equation*}
 \begin{split}
 g_2(\pi_{*}W,\pi_{*}W)\frac{d(h\circ \Omega)}{ds}&=g_2(\mathcal{A}_{\psi\pi_{*}W}\pi_{*}W,\pi_{*}Z)-g_2(\nabla_W^{\pi\perp}\psi(\pi_{*}W)+\nabla_U^{\pi\perp}\psi(\pi_{*}W),\mathcal{C}U)\\&-\alpha\eta(U) g_1(W,Z),
 \end{split}
 \end{equation*}
 where $U\in \Gamma((range\pi_{*})^\perp)$ and $W,Z\in \Gamma((ker\pi_{*})^\perp)$ such that $\pi_{*}Z=\mathcal{B}U$. Also $\pi_{*}W$ and $U$ are vertical and horizontal part of $\dot{\Omega}(s)$ respectively and $h$ is a smooth function on $B$. 
\end{corollary} 
\begin{corollary}
 Let $\pi: M^m\rightarrow B^b$ be an anti-invariant Riemannian map from a Riemannian manifold $(M^m,g_1)$ to a trans-Sasakian manifold $(B^b,g_2,\psi,\eta,\xi)$ of type $(0,\beta)$ with horizontal Reeb vector field $\xi$. Let $\gamma$ and $\Omega$ be geodesic on $M$ and $B$ respectively. Then $\pi$ is a Clairaut anti-invariant Riemannian map with $r=e^h$ if and only if
 \begin{equation*}
 \begin{split}
 g_2(\pi_{*}W,\pi_{*}W)\frac{d(h\circ \Omega)}{ds}&=g_2(\mathcal{A}_{\psi\pi_{*}W}\pi_{*}W,\pi_{*}Z)-g_2(\nabla_W^{\pi\perp}\psi(\pi_{*}W)+\nabla_U^{\pi\perp}\psi(\pi_{*}W),\mathcal{C}U)\\&-\beta\eta(U)||\psi U||^2,
 \end{split}
 \end{equation*}
 where $U\in \Gamma((range\pi_{*})^\perp)$ and $W,Z\in \Gamma((ker\pi_{*})^\perp)$ such that $\pi_{*}Z=\mathcal{B}U$. Also $\pi_{*}W$ and $U$ are vertical and horizontal part of $\dot{\Omega}(s)$ respectively and $h$ is a smooth function on $B$. 
\end{corollary}
\begin{corollary}
 Let $\pi: M^m\rightarrow B^b$ be an anti-invariant Riemannian map from a Riemannian manifold $(M^m,g_1)$ to a trans-Sasakian manifold $(B^b,g_2,\psi,\eta,\xi)$ of type $(0,0)$ with horizontal Reeb vector field $\xi$. Let $\gamma$ and $\Omega$ be geodesic on $M$ and $B$ respectively. Then $\pi$ is a Clairaut anti-invariant Riemannian map with $r=e^h$ if and only if
  \begin{equation*} 
  g_2(\pi_{*}W,\pi_{*}W)\frac{d(h\circ \Omega)}{ds}=g_2(\mathcal{A}_{\psi\pi_{*}W}\pi_{*}W,\pi_{*}Z)-g_2(\nabla_W^{\pi\perp}\psi(\pi_{*}W)+\nabla_U^{\pi\perp}\psi(\pi_{*}W),\mathcal{C}U),
  \end{equation*}
  where $U\in \Gamma((range\pi_{*})^\perp)$ and $W,Z\in \Gamma((ker\pi_{*})^\perp)$ such that $\pi_{*}Z=\mathcal{B}U$. Also $\pi_{*}W$ and $U$ are vertical and horizontal part of $\dot{\Omega}(s)$ respectively and $h$ is a smooth function on $B$.
\end{corollary}
\begin{theorem}
 Let $\pi: M^m\rightarrow B^b$ be a Clairaut anti-invariant Riemannian map from a Riemannian manifold $(M^m,g_1)$ to a trans-Sasakian manifold $(B^b,g_2,\psi,\eta,\xi)$ having horizontal Reeb vector field $\xi$ with $r=e^h$. Then either $dim(range\pi_{*})=1$ or $h$ is constant in $\psi (range\pi_{*})$.
\end{theorem}
\begin{proof}
Since $\pi$ is a Clairaut anti-invariant Riemannian map admitting horizontal Reeb vector field with $r=e^h$, we have
\begin{equation}\label{l}
(\nabla\pi_{*})(W,Z)=-g(W,Z)\nabla^Bh
\end{equation}
for any $W,Z\in \Gamma(ker\pi_{*})^\perp$. Taking inner product with $\psi\pi_{*}Y \in \Gamma((range\pi_{*})^\perp)$ and using \eqref{5}, we get
\begin{equation*}
g_2(\overset{B}{\nabla^{\pi}_{W}}\pi_{*}Z,\psi\pi_{*}Y)=-g_1(W,Z)g_2(\nabla^Bh,\psi \pi_{*}Y).
\end{equation*}
Also, from above equation we have
\begin{equation}\label{k}
g_2(\overset{B}{\nabla^{\pi}_{W}}\psi\pi_{*}Y,\pi_{*}Z)=g_1(W,Z)g_2(\nabla^Bh,\psi \pi_{*}Y).
\end{equation}
Since $B$ is a trans-Sasakian manifold, using \eqref{4} in \eqref{k}, we get
\begin{equation}\label{m}
g_2(\overset{B}{\nabla^{\pi}_{W}}\pi_{*}Y,\psi\pi_{*}Z)=-g_1(W,Z)g_2(\nabla^Bh,\psi \pi_{*}Y).
\end{equation}
Again, using \eqref{l}, we have
\begin{equation}\label{n}
g_2(\overset{B}{\nabla^{\pi}_{W}}\pi_{*}Y,\psi\pi_{*}Z)=-g_1(W,Y)g_2(\nabla^Bh,\psi \pi_{*}Z),
\end{equation}
equating \eqref{m} and \eqref{n}, we obtain
\begin{equation*}
g_1(W,Z)g_2(\nabla^Bh,\psi \pi_{*}Y)=g_1(W,Y)g_2(\nabla^Bh,\psi \pi_{*}Z).
\end{equation*}
Putting $W=Z$ in above equation, we get
\begin{equation}\label{o}
||W||^2g_2(\nabla^Bh,\psi \pi_{*}Y)=g_1(W,Y)g_2(\nabla^Bh,\psi\pi_{*}W).
\end{equation}
Interchanging $W$ and $Y$ in above equation, we have
\begin{equation}\label{p}
||Y||^2g_2(\nabla^Bh,\psi \pi_{*}W)=g_1(W,Y)g_2(\nabla^Bh,\psi\pi_{*}Y).
\end{equation}
From \eqref{o} and \eqref{p}, we obtain
\begin{equation}\label{q}
g_2(\nabla^Bh,\psi \pi_{*}W)\Big[1-\dfrac{g_1(W,Y)g_1(W,Y)}{||W||^2||Y||^2}\Big]=0.
\end{equation}
From \eqref{q}, we conclude that either $dim((ker\pi_{*})^\perp)=1$ or $h$ is constant in $\psi\pi_{*}W$. Since there is linear isometry between $(ker\pi_{*})^\perp$ and $range\pi_{*}$. Hence we have the theorem.
\end{proof}

\begin{theorem}
 Let $\pi: M^m\rightarrow B^b$ be a Clairaut anti-invariant Riemannian map from a Riemannian manifold $(M^m,g_1)$ to a trans-Sasakian manifold $(B^b,g_2,\psi,\eta,\xi)$ having horizontal Reeb vector field $\xi$. If $dim (range\pi_{*})>1$, then $range\pi_{*}$ is minimal.
\end{theorem} 
 \begin{proof}
 Let $Y\in\Gamma((ker\pi_{*})^\perp)$, then we have
 \begin{equation}\label{s}
 (\nabla\pi_{*})(Y,Y)=g(Y,Y)H_2.
 \end{equation}
 If $\pi_{*}Z\in\Gamma((range\pi_{*})^\perp)$, using \eqref{5}, above eqution can be written as
\begin{eqnarray}\label{r}
g_2(\pi_{*}Y,\nabla_Y^\pi\psi\pi_{*}Z)=-g_1(Y,Y)g_2(H_2,\psi\pi_{*}Z). 
 \end{eqnarray}
 Since $B$ ia a trans-Sasakian manifold, simplifying \eqref{r} we have
 \begin{eqnarray}\label{v}
 g_2(\psi\pi_{*}Y,\nabla_Y^\pi\pi_{*}Z)=g_1(Y,Y)g_2(H_2,\psi\pi_{*}Z). 
  \end{eqnarray}
  Again, from \eqref{s} and \eqref{v} we get
  \begin{equation}\label{t}
 g_1(Y,Z)g_2(H_2,\psi\pi_{*}Y)=g_1(Y,Y)g_2(H_2,\psi\pi_{*}Z).  
  \end{equation}
  Interchanging $Y$ and $Z$, we have
  \begin{equation}\label{u}
  g_1(Y,Z)g_2(H_2,\psi\pi_{*})Z=g_1(Z,Z)g_2(H_2,\psi\pi_{*}Y).
  \end{equation}
 Since $dim (range\pi_{*})>1$, from \eqref{t} and \eqref{u} we conclude the required result.   
 \end{proof}
 
 \begin{theorem} 
 Let $\pi: M^m\rightarrow B^b$ be a Clairaut anti-invariant Riemannian map from a Riemannian manifold $(M^m,g_1)$ to a trans-Sasakian manifold $(B^b,g_2,\psi,\eta,\xi)$ having horizontal Reeb vector field $\xi$. If $range\pi_{*}$ is integrable, then $g_2(\nabla_W^{\pi\perp}\psi(\pi_{*}Y)-\nabla_Y^{\pi\perp}\psi(\pi_{*}W),\mathcal{C}U)=0,$
 where $W,Y\in\Gamma((ker\pi_{*})^\perp)$ and $U\in\Gamma ((range\pi_{*})^\perp),$
 \end{theorem}
 \begin{proof}
 Let $W,Y \in \Gamma (ker \pi_{*}^{\perp})$ and $U\in \Gamma ((range \pi_{*})^\perp)$, we have
 \begin{equation}\label{w}
 g_2([\pi_{*}W,\pi_{*}Y],U)=g_2(\nabla^B_{\pi_{*}W}\pi_{*}Y-\nabla^B_{\pi_{*}Y}\pi_{*}W,U).
 \end{equation}
Since $B$ is a trans-Sasakian manifold, from \eqref{b},\eqref{4} and \eqref{w}, we get 
   \begin{equation*}
   g_2([\pi_{*}W,\pi_{*}Y],U)=g_2(\nabla^B_{\pi_{*}W}\psi(\pi_{*}Y)-\nabla^B_{\pi_{*}Y}\psi(\pi_{*}W),\psi U).
   \end{equation*}
Using \eqref{6} and \eqref{8} in above equation, we obtain
 \begin{equation}\label{y}
 g_2([\pi_{*}W,\pi_{*}Y],U)=-g_2(\mathcal{A}_{\psi(\pi_{*}Y)}\pi_{*}W,\mathcal{B}U)+g_2(\mathcal{A}_{\psi(\pi_{*}W)}\pi_{*}Y,\mathcal{B}U)+g_2(\nabla_W^{\pi\perp}\psi(\pi_{*}Y)-\nabla_Y^{\pi\perp}\psi(\pi_{*}W),\mathcal{C}U).
 \end{equation}
Assuming $Z\in\Gamma((ker\pi_{*})^\perp)$ such that $\pi_{*}Z=\mathcal{B}U$ and using \eqref{10}, \eqref{y} can be rewritten as
 \begin{equation*}
 \begin{split}
  g_2([\pi_{*}W,\pi_{*}Y],U)&=-g_2(\psi(\pi_{*}Y),(\nabla\pi_{*})(W,Z))+g_2(\psi(\pi_{*}W),(\nabla\pi_{*})(Y,Z))\\&
  +g_2(\nabla_W^{\pi\perp}\psi(\pi_{*}Y)-\nabla_Y^{\pi\perp}\psi(\pi_{*}W),\mathcal{C}U).
 \end{split}
  \end{equation*}
  Since $\pi$ is a Clairaut Riemannian map, using Definition 2.2. in above equation, we get
   \begin{equation}
   \begin{split}\label{z}
    g_2([\pi_{*}W,\pi_{*}Y],U)&=g_2(\psi(\pi_{*}Y),\nabla^Bh)[g_1(W,Z)-g_1(Y,Z)]\\&
    +g_2(\nabla_W^{\pi\perp}\psi(\pi_{*}Y)-\nabla_Y^{\pi\perp}\psi(\pi_{*}W),\mathcal{C}U).
   \end{split}
    \end{equation}
  Since $dim(range\pi_{*})>1$, using Theorem 3.3 in \eqref{z}, we get the required result. 
 \end{proof}
 
 \begin{theorem}
 Let $\pi: M^m\rightarrow B^b$ be a Clairaut anti-invariant Riemannian map from a Riemannian manifold $(M^m,g_1)$ to a trans-Sasakian manifold $(B^b,g_2,\psi,\eta,\xi)$ having horizontal Reeb vector field $\xi$. Then $(range\pi_{*})^\perp$ is integrable.
 \end{theorem}
 \begin{proof}
Let $U,V\in\Gamma((range\pi_{*})^\perp)$ and $W\in\Gamma(range\pi_{*})$, then we can write 
 \begin{equation}
 g_2([U,V],W)=g_2(\nabla_UV-\nabla_VU,W).
 \end{equation}
 Since, $(range\pi_{*})^\perp$ is a totally geodesic distribution so we have the required result.
 \end{proof}
 
 \begin{theorem}
 Let $\pi: M^m\rightarrow B^b$ be a Clairaut anti-invariant Riemannian map from a Riemannian manifold $(M^m,g_1)$ to a trans-Sasakian manifold $(B^b,g_2,\psi,\eta,\xi)$ having horizontal Reeb vector field $\xi$ and $dim(range\pi_{*})>1$. Then, $\pi$ is harmonic if and only if $ker\pi_{*}$ is minimal.
 \end{theorem}
\begin{proof}
Let $\{Z_i\}_{i=1}^{r}$ and $\{Z_i\}_{i=r+1}^{m}$ be orthonormal basis of $ker\pi_{*}$ and $(ker\pi_{*})^\perp$ respectively, then we have
\begin{equation}\label{21}
\begin{split}
trace(\nabla \pi_{*})&=\sum_{i=1}^{r}(\nabla \pi_{*})(Z_i,Z_i)+\sum_{i=r+1}^{m}(\nabla \pi_{*})(Z_i,Z_i)\\
&=\sum_{i=1}^{r}(\nabla \pi_{*})(Z_i,Z_i)+\sum_{i=r+1}^{m}g_2((\nabla \pi_{*})(Z_i,Z_i),\pi_{*}Z_i)\pi_{*}Z_i\\&
+\sum_{i=r+1}^{m}\sum_{j=1}^{s}g_2((\nabla \pi_{*})(Z_i,Z_i),\mu_j)\mu_j+\sum_{i=r+1}^{m}g_2((\nabla \pi_{*})(Z_i,Z_i),\psi(\pi_{*}Z_i))\psi(\pi_{*}Z_i),
\end{split}
\end{equation}
where $\{\pi_{*}Z_i\}_{i=r+1}^{m}$ and $\{\mu_j\}_{j=1} ^{s}$ are orthonormal basis of $\Gamma(range\pi_{*})$ and $\Gamma(\mu)$ respectively, and $b=2m+s$.\\ 
Using lemma 2.1 and \eqref{5} in \eqref{21}, we get
\begin{equation}
\begin{split}
trace(\nabla \pi_{*})&=\sum_{i=1}^{r}(\nabla^\pi_{Z_i}\pi_{*}Z_i-\pi_{*}(\nabla_{Z_i}^MZ_i))+\sum_{i=r+1}^{m}\sum_{j=1}^{s}g_2((\nabla \pi_{*})(Z_i,Z_i),\mu_j)\mu_j\\&
+\sum_{i=r+1}^{m}g_2(\nabla^\pi_{Z_i}\pi_{*}Z_i,\psi(\pi_{*}Z_i))\psi(\pi_{*}Z_i).
\end{split}
\end{equation}Since $B$ is a trans-Sasakian manifold, using \eqref{4} in above equation, we have
\begin{equation}\label{23}
\begin{split}
trace(\nabla\pi_{*})&=-\sum_{i=1}^{r}\pi_{*}(\nabla_{Z_i}^MZ_i))+\sum_{i=r+1}^{m}\sum_{j=1}^{s}g_2((\nabla \pi_{*})(Z_i,Z_i),\mu_j)\mu_j\\&
+\sum_{i=r+1}^{m}[g_2(-\psi\nabla^\pi_{Z_i}\psi(\pi_{*}Z_i),\psi(\pi_{*}Z_i))\psi(\pi_{*}Z_i)-\beta g_2(\pi_{*}Z_i,\pi_{*}Z_i)g_2(\xi,\psi(\pi_{*}Z_i))\psi(\pi_{*}Z_i)],\\&
%-\sum_{i=1}^{r}\pi_{*}(\nabla_{Z_i}^MZ_i))+\sum_{i=r+1}^{m}\sum_{j=1}^{s}g_2((\nabla \pi_{*})(Z_i,Z_i),\mu_j)\mu_j\\&
%-\sum_{i=r+1}^{m}g_2(\nabla^\pi_{Z_i}\psi(\pi_{*}Z_i),\pi_{*}Z_i)\psi(\pi_{*}Z_i)\\&
=-\sum_{i=1}^{r}\pi_{*}(\nabla_{Z_i}^MZ_i))+\sum_{i=r+1}^{m}\sum_{j=1}^{s}g_2((\nabla \pi_{*})(Z_i,Z_i),\mu_j)\mu_j\\&
+\sum_{i=r+1}^{m}g_2(\nabla^\pi_{Z_i}\pi_{*}Z_i,\psi(\pi_{*}Z_i))\psi(\pi_{*}Z_i).
\end{split}
\end{equation}
Further, using \eqref{7} and \eqref{22} in \eqref{23}, we get
\begin{equation}
\begin{split}
trace(\nabla \pi_{*})&=-r\pi_{*}(\varrho^{\mathcal{V}})+\sum_{i=r+1}^{m}\sum_{j=1}^{s}g_2(H_2g_1(Z_i,Z_i),\mu_j)\mu_j\\&+(m-r)\sum_{i=r+1}^{m}g_2(H_2,\psi(\pi_{*}Z_i))\psi(\pi_{*}Z_i),
\end{split}
\end{equation}  
where, $\varrho^{\mathcal{V}}$ is mean curvature of $ker\pi_{*}$. Since $dim(range\pi_{*})>1$, from theorem 3.4 and above equation, we get
\begin{equation}
trace(\nabla \pi_{*})=-r\pi_{*}(\varrho^{\mathcal{V}}).
\end{equation}
Thus, $\pi$ is harmonic if and only if $ker\pi_{*}$ is minimal.
\end{proof}  
 
\begin{example}
Let $\pi: M\rightarrow B$ be a smooth map defined as $$\pi(x,y,z)=(0,x+y,0),$$
where $M=\{(x,y,z)\in \mathbb{R}^3,x,y,z\neq0\}$ is a Riemannian manifold with Riemannian metric $$g_1=\frac{1}{4}\begin{pmatrix}
\frac{3}{2}&\frac{1}{2}&0\\
\frac{1}{2}&\frac{3}{2}&0\\
0&0&1
\end{pmatrix}$$ on $M$ and $B=\{(x,y,z)\in\mathbb{R}^3,z\neq0\}$ is a trans-Sasakian manifold with contact structure given by Example 2.1., then we have
$$(ker\pi_{*})=span\big\{e_1-e_2,e_3\big\},$$
$$(ker\pi_{*})^\perp=span\{Z=e_1+e_2\},$$ where $e_i$ are standard basis vector fields on $M$. Also, by simple computation it is easy to see that 
$$(range\pi_{*})=span\big\{\pi_{*}Z=E_1=2\frac{\partial}{\partial v}\big\}$$
$$(range\pi_{*})^\perp=span\big\{E_2=2(\frac{\partial}{\partial u}+v\frac{\partial}{\partial w}),E_3=2\frac{\partial}{\partial w}=\xi\big\}$$
and $\psi E_1=E_2$ with $g_1(Z,Z)=g_2(\pi_{*}Z,\pi_{*}Z)$. Thus, $\pi$ is an anti-invariant Riemannian map.\\
In order to show that the defined map is Clairaut Riemannian map we find a smooth function $h$ satisfying equation $(\nabla\pi_{*})(Z,Z)=-g(Z,Z)\nabla^Bh$. Here, $(\nabla\pi_{*})(Z,Z)=0$, $g_1(Z,Z)=1$, thus by taking constant $h$, we can verify $(\nabla\pi_{*})(Z,Z)=-g(Z,Z)\nabla^Bh$. Also from \eqref{4}, we have $\alpha=1, \beta=0$. Hence $\pi$ is a Clairaut anti-invariant Riemannian map from Riemannian manifold to trans-Sasakian manifold of type $(1,0).$
\end{example} 
\begin{example}
Let $\pi:M\rightarrow B$ be a Riemannian map defined as
$$\pi(x,y,z)=(0,\frac{x-y}{\surd 2},0), $$  where $M$ is a Riemannian manifold and $B$ is a Sasakian manifold. Since, there is a linear isometry between $(ker\pi_{*})^\perp$ and $(range\pi_{*})$, we can define a smooth function $k$ between these distributions and then pullback $k^*$ of that function on $(ker\pi_{*})^\perp$ in terms of $B's$ co-ordinate system such that $k^*((ker\pi_{*})^\perp)=e^{-w}$, then extend this function on whole $TM$. Now, we can define a global frame $\{e_1,e_2,e_3\}$ with $e_1=e^{-w}\frac{\partial}{\partial x},e_2=e^{-w}\frac{\partial}{\partial y},e_3=\frac{\partial}{\partial z}$ and a Riemannian metric $g_1$ on $M$ such that $g_1(x,y,z)=e^{2w}dx^2+e^{2w}dy^2+dz^2$, whereas $B=\{(u,v,w)\in \mathbb{R}^3|v,w\neq0\}$ is equipped with a contact metric structure $(g_2, \psi,\eta, \xi),$ given by\\
 $$ g_2=(e^{2w}+v^2)du^2+e^{2w}dv^2+(-2v)dvdw+dw^2,~~\psi=\begin{pmatrix}
 			 				0&1&0\\
 			            	-1&0&0\\
 			          		0&v&0
 							 \end{pmatrix},~~ 
 					 	\eta=(dw-vdu), ~~\xi=\frac{\partial}{\partial w}$$  and  $\{E_1,E_2,E_3\}$ is a global frame on $B$, defined as $E_1=e^{-w}\frac{\partial}{\partial v}, E_2=\psi E_1=e^{-w}(\frac{\partial}{\partial u}+v\frac{\partial}{\partial w}), E_3=\frac{\partial}{\partial w}=\xi$.\\
 Then, by simple calculation, we get
 $$(ker\pi_{*})^\perp=span\big\{Z=\frac{1}{\surd 2}(e_1-e_2)=\frac{e^{-w}}{\surd 2}\big(\frac{\partial}{\partial x}-\frac{\partial}{\partial y}\big)\big\},$$
 $$range\pi_{*}=span\big\{\pi_{*}Z=e^{-w}\frac{\partial}{\partial v}=E_1\big\},$$
 $$(range\pi_{*})^\perp=span\big\{E_2=e^{-w}\big(\frac{\partial}{\partial u}+\frac{\partial}{\partial w}\big), E_3=\frac{\partial}{\partial w}=\xi\big\}.$$	
Also, $g_1(Z,Z)=g_2(\pi_{*}Z,\pi_{*}Z)=1$ and $\psi(range\pi_{*})\subset((range\pi_{*})^\perp)$, therefore $\pi$ is an anti-invariant Riemannian map. In order to prove that $\pi$ is a Clairaut map, we must have $\nabla \pi_{*}(Z,Z)=-g_1(Z,Z)\nabla^Bh.$ Here, by some computation, it is easy to see that $g_1(Z,Z)=1$ and $\nabla \pi_{*}(Z,Z)=-E_3-ve^{-w}E_2$, therefore, we have $\nabla^Bh=E_3+ve^{-w}E_2 $. For a smooth function $h$, the value of $\nabla^Bh$ with respect to $g_2$ is given by $\nabla^Bh=\big(e^{-w}\frac{\partial h}{\partial u}+ve^{-w}\frac{\partial h}{\partial w}\big)E_2+v^2\frac{\partial h}{\partial w}E_3$, which implies that $h=\frac{1}{ve^w}.$ Also from \eqref{4}, we have $\alpha=\frac{1}{2}e^{-2w},\beta=1$. Hence $\pi$ is a Clairaut anti-invariant Riemannian map to trans-Sasakian manifold of type $(\frac{1}{2}e^{-2w},1)$.     				 	
\end{example}
 
 	\section{Acknowledgments}
 The first author is thankful to UGC for providing financial assistance in terms of MANF scholarship vide letter with UGC-Ref. No. 1844/(CSIR-UGC NET JUNE 2019). The second author is thankful to DST Gov. of India for providing financial support in terms of DST-FST label-I grant vide sanction number SR/FST/MS-I/2021/104(C).

 \end{document}